\documentclass{article}
\usepackage[a4paper, total={6in, 8in}]{geometry}
\usepackage{setspace}
\usepackage[usenames,dvipsnames,svgnames,table]{xcolor}
\usepackage{amsmath, amssymb, amsthm}
\usepackage{mathpartir}
\usepackage{proof, microtype, hyperref}
\usepackage{sectsty}
\usepackage{cite}

\newcommand{\bbR}{\mathbb{R}}

\newtheorem{theorem}{Theorem}[section]
\newtheorem{claim}[theorem]{Claim}
\newtheorem{theorem?}[theorem]{Theorem?}

\newtheorem{lemma}[theorem]{Lemma}

\newtheorem{corr}[theorem]{Corollary}

\theoremstyle{definition}

\newtheorem{note}[theorem]{Note}

\newtheoremstyle{case}{}{2pt}{}{}{\scshape}{:}{ }{}
\theoremstyle{case}

\begin{document}
\title{Note on $\mathsf{TD} + \mathsf{DC}_\bbR$ implying $\mathsf{AD}^{L(\bbR)}$}

\author{Sean Cody}

\maketitle
\onehalfspacing
\section{Introduction}
There are two known proofs that $\mathsf{TD}+\mathsf{DC}_\bbR$ imply $\mathsf{AD}^{L(\bbR)}$ both due to Woodin, but currently no published accounts of either exist. The later proof involves proving the stronger result of Suslin determinacy from Turing determinacy + $\mathsf{DC}_\bbR$ directly~\cite{woodin:TDsus}. Combining that with Kechris and Woodin's theorem that Suslin determinacy in $L(\bbR)$ implies $\mathsf{AD}^{L(\bbR)}$~\cite{KWReflection}, the desired result becomes an immediate corollary.

Woodin's original proof uses an early version of the core model induction (CMI) technique. Through the work of many set theorists, the CMI has been developed into a proper framework for proving determinacy results from non-large cardinal hypotheses such as generic elementary embeddings, forcing axioms, and the failure of fine-structural combinatorial principles. The technique as it is understood in $L(\bbR)$-like models (i.e., $L(\bbR^g)$ where $\bbR^g$ are the reals of a symmetric collapse) can be seen as an inductive method by which one proves $J_\alpha(\bbR)\models\mathsf{AD}$ for all $\alpha$. An introduction to this as well as all terminology used in this paper can be found in Schindler and Steel's book~\cite{CMIBook}.

This paper aims to prove that $\mathsf{TD} + \mathsf{DC}_\bbR$ implies $\mathsf{AD}^{L(\bbR)}$ using modern perspectives on the core model induction in $L(\bbR)$. The key lemma is a modification of a well known theorem of Kechris and Solovay to work in the $\mathsf{TD}$ context. Utilizing the proof of the witness dichotomy it is sufficient to just prove the $J\mapsto M_1^{\#,J}$ step of the core model induction.
\subsection*{Acknowledgements}
The author would like to thank Woodin for giving permission to write up this result. He would also like to thank Schindler and Steel for reading over the proof and giving several helpful suggestions.
\section{Kechris-Solovay Theorem}
The following is our primary lemma which is a modification on a theorem of Kechris and Solovay~\cite{RSDH}.
\begin{lemma}\label{KS}
Assume $\mathsf{\mathsf{TD}}$. For any set of ordinals $S$, on a Turing cone $\mathcal C$ the following holds for $x\in \mathcal C$
$$L[S,x]\models\mathsf{OD}_{{S}}\text{-}\mathsf{\mathsf{TD}}$$
\end{lemma}
\begin{proof}
Assume for a contradiction that there is no cone of reals $\mathcal{C}$ such that $L[S,x]\models\mathsf{OD}_{S}\text{-}\mathsf{\mathsf{TD}}$. Then we can define, on a cone, the map $x\mapsto A_x$ where $A_x$ is the least $\mathsf{OD}_{S}^{L[S,x]}$ $\equiv_T$-invariant subset of $\mathbb R$ which doesn't contain a Turing cone and whose complement does not contain a Turing cone in $L[S,x]$. Notice that $A_x$ only depends on the $S$-constructibility degree of $x$.\\[2mm]
It is clear from the last observation that the set $\{x\in\mathbb R\mid x\in A_x\}$ is $\equiv_T$-invariant and is well-defined on $\mathcal C$. Suppose that this set contains a Turing cone $\mathcal C'$. Consider some arbitrary $y\in \mathcal C\cap \mathcal C'$, if $w\geq_{T} z\geq_{T} y$ and $w\in L[S,z]$ then we have that $w\in A_w = A_z$. So $A_z$ contains a Turing cone in $L[S,z]$. We reach a similar conclusion if we assume that $\{x\in\mathbb R\mid x\not\in A_x\}$ contains a Turing cone. Contradiction.
\end{proof}
\noindent Ordinal (Turing) determinacy has the following well-known consequence :
\begin{corr}\label{KS2} $\mathsf{\mathsf{HOD}}_{{S}}^{L[S,x]}\models\omega_1^{L[S,x]}\text{ is measurable}$ for a cone of $x$.
\end{corr}
\begin{proof}[Proof outline]
Working in $L[S,x]$, assume that $\mathsf{OD}_{S}\text{-}\mathsf{\mathsf{TD}}$ holds. We can do the standard construction of Martin's measure $\mu$ on $P(\omega_1^V)\cap\mathsf{OD}_S$. It is easy to see that $\mathsf{\mathsf{HOD}}^{L[S,x]}_{S}\cap\mu$ is a countably complete ultrafilter on $\omega_1^{L[S,x]}$ in $\mathsf{\mathsf{HOD}}^{L[S,x]}_S$.
\end{proof}
\begin{note}
The proofs of Theorem \ref{KS} and Corollary \ref{KS2} don't actually rely on any essential property of $L$ that is not shared by $L^J$ where $J$ is some (hybrid) model operator. Strictly speaking, the $L^J$ variants are what are used in Section 4.
\end{note}
\section[The Existence of Msharp1]{The Existence of $M_1^\#$}
The following consequences of $\Delta^1_2$-$\mathsf{\mathsf{TD}}$ are proven by modifying the analogous arguments from $\Delta^1_2$-$\mathsf{Det}$ in a similar fashion to the core argument of Theorem \ref{KS} then verifying that nothing breaks. As the modifications are relatively straightforward the proofs will not be included.
\begin{itemize}
\item (Martin) Assume $\Delta^1_2$-$\mathsf{TD + DC}$, then $\Pi^1_2\text{-}\mathsf{\mathsf{TD}}$.
\item (Kechris-Solovay) Assume $\Delta^1_2$-$\mathsf{TD + DC}$, then for any real $y$, on a Turing cone $x\in\mathcal C$ the following holds for $x\in \mathcal C$ 
$$L[x,y]\models\mathsf{OD}_{{y}}\text{-}\mathsf{\mathsf{TD}}$$

\item (Consequence of Kechris-Woodin~\cite{KWReflection}) Assume $\Delta^1_2$-$\mathsf{TD + DC}$ and $(\forall x\in\bbR)\: x^\#\text{ exists}$, then $\operatorname{Th}(L[x])$ is fixed on a Turing cone. 
\end{itemize}
We can utilize these three observations to prove the first step in our induction from a weaker hypothesis.
\begin{theorem} Assume $\Delta^1_2$-$\mathsf{TD + DC}$ and $(\forall x\in\bbR)\: x^\#\text{ exists}$, then $M_1^\#$ exists.
\end{theorem}
\begin{proof}
Utilizing we have that $L[x]\models\mathsf{OD}\text{-}\mathsf{\mathsf{TD}}$ on a Turing cone. Let $x$ be the base of such a cone and fix the least $x$-indiscernible $i_0$. We can run the $K^c$ construction in $L[x]$ below $i_0$. Suppose that the $K^c$ construction in $L[x]$ does not reach $M_1^\#$ so $K^L[x]$ exists. Letting $U$ be the ultrafilter on $i_0$ from $x^\#$, then for a $U$-measure one set of $\alpha< i_0$ we have that $L[x]\models (\alpha^+)^K = \alpha^+$. Select such an $\alpha$ and let $z = \left<x,g\right>$ where $g$ is $L[x]$-generic for $Coll(\omega,\alpha)$. Working in $L[z]$, $K$ exists, is inductively definable, and $\omega_1$ is a successor cardinal in $K$. As $z\geq_T x$ we have that $L[z]\models\mathsf{OD}\text{-}\mathsf{\mathsf{TD}}$, therefore $HOD^{L[z]}\models\omega_1\text{ is measurable}$. But as $K^{L[z]}\subseteq HOD^{L[z]}$ we have a contradiction.\\[2mm]
Therefore we have that $M_1^{\#^{L[x]}}$ exists. As $M_1^\#$ is definable it's constant on a cone of $x$, this then gives us the real $M_1^\#$.
\end{proof}
\section[The J maps to MsharpJ1 step]{The $J\mapsto M^{\#,J}_1$ Step}
Following this point on the argument is identical to that of Steel and Schindler. But I will rewrite it (practically verbatim) for the sake of completeness. For the rest of this argument we will assume $\mathsf{ZF} + \mathsf{TD} + \mathsf{DC} + V = L(\bbR)$. Note that every (hybrid) model operator considered in the core model induction relativizes well and determines itself on generic extensions.

\begin{theorem}\label{core model}
Let $a\in\mathbb{R}$ and let $J$ be a (hybrid) model operator that relativizes well and determines itself on generic extensions, and suppose that $$W_x = \left(K^{c,J}(a)\right)^{\mathsf{\mathsf{HOD}}_{a}^{L^J[x]}}$$
constructed with height $\omega_1^V$ exists for a cone of $x$. Then there is a cone of $x$ such that $W_x$ cannot be $\omega_1^V+1$ iterable above $a$ inside $\mathsf{\mathsf{HOD}}_{a}^{L^J[x]}$
\end{theorem}
\begin{proof}
Assume for a contradiction that this is not the case. Then there is an cone $\mathcal{C}$ such that for all $x\in\mathcal C$ we have $L^J[a,x] = L^J[x]$ and
$$\omega_1^{L^J[x]}\text{ is measurable in }\mathsf{\mathsf{HOD}}_{a}^{L^J[x]}$$
and so we can isolate $K^{J}_x = \left(K^{ J}(a)\right)^{\mathsf{\mathsf{HOD}}_{a}^{L^J[x]}}$. Fixing some $x\in\mathcal C$ we will write $K^{ J}$ for $K^{ J}_x$. By cheapo covering and the fact that $L^J[x]$ is a size $<\omega_1^V$ forcing extension of $\mathsf{HOD}^{L^J[x]}_{a}$, we can choose some $\lambda < \omega_1^V$ such that 
$$\lambda^{+K^{ J}} = \lambda^{+L^J[x]}$$
Let $g\in V$ be $Col(\omega,\lambda)$-generic over $L^J[x]$ and let $y\in V$ be a real coding $(g,x)$. Then
$$\omega_1^{L^J[y]} = \lambda^{+K^{J}} = \lambda^{+L^J[x]}$$
As $y\in \mathcal C$ we have that
$$\omega_1^{L^J[y]}\text{ is measurable in }\mathsf{\mathsf{HOD}}^{L^J[y]}_{a}$$
We get a contradiction if we can show that
\begin{claim}
$K^{ J}\in\mathsf{\mathsf{HOD}}^{L^J[y]}_{a}$ 
\end{claim}
\begin{proof}
$K^{ J}$ is still fully iterable inside $L^J[y] = L^J[x][g]$. This means that $K^{ J}$ is the core model above $a$ of $L^J[y]$ in the sense that it is the common transitive collapse of $Def(W,S)$ for any $W,S$ such that $W$ is an $\omega_1^V$-iterable $ J$-weasel and $\omega_1^V$ is $S$-thick. This characterization demonstrates that $K^{ J}\in\mathsf{\mathsf{HOD}}^{L^J[y]}_{a}$.
\end{proof}
\end{proof}
Utilizing this theorem we wish to show $\forall\alpha\;W^*_\alpha$. Suppose that for some fixed critical $\alpha$, $W_\alpha^*$ holds. We wish to show that $W^*_{\alpha+1}$ holds. By the witness dichotomy (this is where we require $\mathsf{\mathsf{DC}}$) this means that we need to see that for all $n<\omega$, $J^n_\alpha$ is total on $\mathbb{R}$. Suppose that $\bbR$ is closed under $J = J^n_\alpha$. We need to see that $\mathbb{R}$ is closed under $J^{n+1}_\alpha$.\\[2mm]
Let us fix $a\in\bbR$. By Theorem \ref{core model}, for any (hybrid) model operator $J$ which relatives well and determines itself on generic extensions there is a cone of $b$ on which $$W_b = \left(K^{c, J}(a)\right)^{\mathsf{\mathsf{HOD}}_{a}^{L^J[b]}}$$
cannot be $\omega_1^V+1$ iterable inside $\mathsf{\mathsf{HOD}}_{a}^{L^J[b]}$. This means we can define the map
$$b\mapsto J_\alpha^{n+1}(a)^b$$
on a cone where $J^{n+1}_\alpha(a)^b$ is a version of $J_\alpha^{n+1}(a)$ from the point of view of $\mathsf{\mathsf{HOD}}_{a}^{L^J[b]}$. Note that in particular, $J^{n+1}_\alpha(a)$ is $\omega_1+1$ iterable in $\mathsf{\mathsf{HOD}}_{a}^{L^J[b]}$.\\[2mm]
Consider the map $f:\mathcal D\to\mathbb{R}$ given by
$$[b]\mapsto\text{a canonical real code for }J^{n+1}_\alpha(a)^b$$
By $\mathsf{\mathsf{TD}}$, for each $n<\omega$ the set $\{b\in\mathbb{R} : n\in f([b])\}$ either contains a cone or is disjoint from a cone. Let $n\in\mathcal P$ iff $n\in f([b])$ on a cone. Then $f([b]) = \mathcal{P}$ on a cone.\\[2mm]
One can see that $\mathcal P$ is $\omega_1$ iterable in $V$ : if $\mathcal T$ is a countable tree on $\mathcal P$ of limit length, then the good branch through $\mathcal T$ is the one picked by the strategies of $\mathsf{\mathsf{HOD}}_{a}^{L^J[b]}$ on a cone. We therefore have found our desired $J^{n+1}_\alpha$.
\begin{note} In the end-of-gaps case we also need to show that the map $J\mapsto M^{\#,J}_0$ is well-defined. A slightly simplified version of the above argument can be made to show that.
\end{note}
\bibliographystyle{unsrt}
\bibliography{TD_draft}

\begin{thebibliography}{1}

\bibitem{woodin:TDsus}
W.~Hugh Woodin.
\newblock Turing {D}eterminacy and {S}uslin sets.
\newblock {\em preprint}, 2021.

\bibitem{KWReflection}
Alexander~S. Kechris and W.~Hugh Woodin.
\newblock Equivalence of partition properties and determinacy.
\newblock {\em Proceedings of the National Academy of Sciences},
  80(6):1783--1786, 1983.

\bibitem{CMIBook}
R.~Schindler and J.~Steel.
\newblock The {C}ore {M}odel {I}nduction; available at
  https://ivv5hpp.uni-muenster.de/u/rds/.

\bibitem{RSDH}
Alexander~S. Kechris and Robert~M. Solovay.
\newblock On the relative consistency strength of determinacy hypothesis.
\newblock {\em Transactions of the American Mathematical Society},
  290(1):179--211, 1985.

\end{thebibliography}
%
%
%

%
%
%
%
%
\end{document}